\documentclass[10pt,reqno]{amsart}
\usepackage{bbm}
\usepackage{mathrsfs}
\usepackage{amsfonts} 
\usepackage[dvipsnames,usenames]{color}
\usepackage{graphicx,latexsym,bm,amsmath,amssymb,verbatim,multicol,lscape}
\usepackage{enumerate}
\newtheorem{thm}{Theorem} [section]
\newtheorem{lem}{Lemma}[section]

\theoremstyle{definition}
\newtheorem{defn}{Definition}[section]
\theoremstyle{remark}
\newtheorem{rem}{Remark}[section]

\numberwithin{equation}{section}
\DeclareMathOperator{\Tr}{Tr}
\DeclareMathOperator{\ind}{ind}
\allowdisplaybreaks

\begin{document}
\title[On the number of zeros of diagonal quartic forms over finite fields]
{On the number of zeros of diagonal quartic forms over finite fields}
\begin{abstract}
Let $\mathbb{F}_q$ be the finite field of $q=p^m\equiv 1\pmod 4$ elements
with $p$ being an odd prime and $m$ being a positive integer. For $c, y
\in\mathbb{F}_q$ with $y\in\mathbb{F}_q^*$ non-quartic, let $N_n(c)$ and
$M_n(y)$ be the numbers of zeros of $x_1^4+...+x_n^4=c$ and
$x_1^4+...+x_{n-1}^4+yx_n^4=0$, respectively. In 1979, Myerson used Gauss sum and
exponential sum to show that the generating function $\sum_{n=1}^{\infty}N_n(0)x^n$
is a rational function in $x$ and presented its explicit expression.
In this paper, we make use of the cyclotomic theory and exponential sums
to show that the generating functions $\sum_{n=1}^{\infty}N_n(c)x^n$ and
$\sum_{n=1}^{\infty}M_{n+1}(y)x^n$ are rational functions in $x$. We also
obtain the explicit expressions of these generating functions.
Our result extends Myerson's theorem gotten in 1979.
\end{abstract}
\author[J.Y. Zhao]{Junyong Zhao}
\address{Mathematical College, Sichuan University, Chengdu 610064, P.R. China}
\email{jyzhao\_math@163.com}
\author[Y.L. Feng]{Yulu Feng}
\address{Mathematical College, Sichuan University, Chengdu 610064, P.R. China}
\email{yulufeng17@126.com}
\author[S.F. Hong]{Shaofang Hong$^*$}
\address{Mathematical College, Sichuan University, Chengdu 610064, P.R. China}
\curraddr{}
\email{sfhong@scu.edu.cn, hongsf02@yahoo.com, s-f.hong@tom.com}
\author[C.X. Zhu]{Chaoxi Zhu}
\address{Mathematical College, Sichuan University, Chengdu 610064, P.R. China}
\email{zhuxi0824@126.com}
\thanks{$^*$S.F. Hong is the corresponding author and was supported
partially by National Science Foundation of China Grant \#11771304.}
\maketitle

\section{Introduction and main results}
Let $p$ be a prime and let $\mathbb{F}_q$ be the finite field
of $q=p^m$ elements with $m$ being a positive integer. Then one can find
an element $g\in\mathbb{F}_q$ such that
$\mathbb{F}_{q}^{*}:=\{g^{i}:i=0,\ 1,\ \dots,\ q-2\}:=\langle g\rangle.$
Such element $g$ is called a {\it generator} of $\mathbb{F}_q^*$.
Suppose that $F(x_1,\dots,x_n)$ is a polynomial with $n$ variables
over $\mathbb{F}_q$. We set $N(F=0)$ to be the number of $n$-tuples
$(x_1,\dots,x_n)$ such that $F(x_1,\dots,x_n)=0$. That is,
\begin{equation*}
N(F=0)=\sharp\{(x_1,\dots,x_n)\in\mathbb{F}_q^n:F(x_1,\dots,x_n)=0\}.
\end{equation*}
Determining the exact value of $N(F=0)$ is an important
topic in number theory and finite field. An explicit formula for $N(F=0)$
is known when $\deg(F)\le 2$ (see \cite{[IR], [LN]}). Generally speaking,
it is difficult to present an explicit formula for $N(F=0)$.
Many authors have studied deeply the $p$-adic behavior of
$N(F=0)$ (see, for example, \cite{[AS],[Ax],[Cao],[Che],[K],[MM],[W2],[War]}).
Studying the formula for $N(F=0)$ has attracted lots
of authors for many years (See, for instance,
\cite{[BEW],[HHZ],[KR],[Sc],[W1]} and \cite{[We]} to \cite{[Wo2]}).

A diagonal form over $\mathbb{F}_q$ is an equation of type
\begin{equation*}
a_1x_1^{e}+\cdots+a_n x_n^{e}=c
\end{equation*}
with $e$ being a positive integer, coefficients
$a_1, ..., a_n \in\mathbb{F}_q^*$ and $c\in\mathbb{F}_q$.
If $e=1$, then
\begin{equation*}
N(a_1x_1+\dots+a_nx_n=c)=q^{n-1}.
\end{equation*}
For $e=2$, there is an explicit formula for
$N(a_1x_1^2+ \dots +a_nx_n^2=c)$ in \cite{[LN]}.
For the case $a_1= \dots =a_n=1$, we denote by
$N_n^{(e)}(c)=N(x_1^e+\cdots+x_n^e=c)$.
In 1977, Chowla, Cowles and Cowles \cite{[CCC]}
presented the expression of the generating function
$\sum_{n=1}^{\infty}N_n^{(3)}(0) x^n$
over $\mathbb{F}_p$ with $p\equiv 1 \pmod 3$ as follows:
\begin{equation*}
\sum_{n=1}^{\infty}N_n^{(3)}(0) x^n
=\frac{x}{1-px}+\frac{x^2(p-1)(2+dx)}{1-3px^2-pdx^3},
\end{equation*}
where $d$ is uniquely determined by
$
4p=d^2+27b^2,\ \text{and} \ d\equiv 1\ (\bmod \ 3).
$
In 1979, Myerson \cite{[My1]} extended the Chowla-Cowles-Cowles
theorem from the prime finite field $\mathbb{F}_p$ to the
general finite field $\mathbb{F}_q$. Recently, Hong and Zhu
\cite{[HZ]} proved that the generating function
$\sum_{n=1}^{\infty}N_n^{(3)}(c)x^n$ is a rational
function and also presented its explicit expression.

Throughout this paper, for brevity, we set
$N_n(c):=N_n^{(4)}(c)$. In \cite{[My1]}, Myerson
also showed that the generating function
$\sum_{n=1}^{\infty}N_n(0)x^n$ over $\mathbb{F}_q$ with
$q\equiv 1\pmod 4$ is rational in $x$ and
\begin{equation}\label{eq1.1}
\sum_{n=1}^{\infty}N_n(0) x^n=
\begin{cases}
\frac{x}{1-qx}+\frac{(q-1)x^2(3-6sx-(q-4s^2)x^2)}
{1-6qx^2+8qsx^3+(q^2-4qs^2)x^4}, & \mbox{if }\ q\equiv 1\ \pmod 8,\\
\frac{x}{1-qx}-\frac{(q-1)x^2(1+6sx+(9q-4s^2)x^2)}
{1+2qx^2+8qsx^3+(9q^2-4qs^2)x^4}, & \mbox{if }\ q\equiv 5\ \pmod 8,
\end{cases}
\end{equation}
where $s$ is uniquely determined by $q=s^{2}+4t^{2}$,
$s\equiv 1\ \pmod 4$, and if $p\equiv 1\pmod 4$,
then $\gcd(s, p)=1$.

In this paper, we address the problem of investigating the
rationality of the generating functions of the sequence
$\{N_n(c)\}_{n=1}^\infty$ for any $c\in {\mathbb F}_q^*$.
We will make use of the cyclotomic theory and exponential sums
to study the generating functions $\sum_{n=1}^{\infty}N_n(c)x^n$.
We show that the generating function $\sum_{n=1}^{\infty}N_n(c)x^n$
is a rational function and also arrive at its explicit
expression. The first main result of this paper can be
stated as follows.

\begin{thm}\label{thm1.2}
Let $\mathbb{F}_q$ be the finite field of $q=p^m$ elements with $p$
being an odd prime and $c\in\mathbb{F}_q$.

{\rm (i).} If $q\equiv 1\pmod 4$ and
$c\in\mathbb{F}_q^*=\langle g\rangle$, then
\begin{align*}
\sum_{n=1}^{\infty}N_n(c)x^n=
\begin{cases}
\frac{x}{1-qx}+\frac{6sx^3+(q-4s^2)x^4+B_1(c,x)}{1-6qx^2+8qsx^3+(q^2-4qs^2)x^4},
& \mbox{if } q\equiv 1\ \pmod 8,\\
\frac{x}{1-qx}+\frac{6sx^3+(9q-4s^2)x^4+B_2(c,x)}{1+2qx^2+8qsx^3+(9q^2-4qs^2)x^4},
& \mbox{if } q\equiv 5\ \pmod 8,
\end{cases}
\end{align*}
with
\begin{small}
\begin{align}\label{eq1.2}
&B_1(c,x):=
\begin{cases}
3x-(6s+3)x^2+(-q+4s^2)x^3, & \mbox{if } \ind_g(c)\equiv 0\ \pmod 4,\\
-x+(2s+8t-3)x^2+(-q-8st)x^3, & \mbox{if } \ind_g(c)\equiv 1\ \pmod 4,\\
-x+(2s-3)x^2+(-q+16t^2)x^3, & \mbox{if } \ind_g(c)\equiv 2\ \pmod 4,\\
-x+(2s-8t-3)x^2+(-q+8st)x^3, & \mbox{if } \ind_g(c)\equiv 3\ \pmod 4
\end{cases}
\end{align}
\end{small}
and
\begin{small}
\begin{align}\label{eq1.3'}
&B_2(c,x):=
\begin{cases}
3x+(2s+1)x^2+(3q-4s^2)x^3, & \mbox{if } \ind_g(c)\equiv 0\ \pmod 4,\\
-x+(2s-8t+1)x^2+(3q+8st)x^3, & \mbox{if } \ind_g(c)\equiv 1\ \pmod 4,\\
-x+(-6s+1)x^2+(-5q-16t^2)x^3, & \mbox{if } \ind_g(c)\equiv 2\ \pmod 4,\\
-x+(2s+8t+1)x^2+(3q-8st)x^3, & \mbox{if } \ind_g(c)\equiv 3\ \pmod 4,
\end{cases}
\end{align}
\end{small}
where if $p\equiv 1\ \pmod 4$, then
$s$ and $t$ are uniquely determined by
$$q=s^{2}+4t^{2}, \ \gcd(p,s)=1, s\equiv 1\ \pmod 4,
2t\equiv sg^{\frac{3(q-1)}{4}}\ \pmod p,$$
and if $p\equiv 3\ \pmod 4$, then $s=(-p)^{\frac{m}{2}}$ and $t=0$.

{\rm (ii).} If $q\equiv 3\pmod4$, then
\begin{align*}
\sum_{n=1}^{\infty}N_n(c)x^n=
\begin{cases}
\frac{x}{1-qx}+\frac{(1-q)x^2}{1+qx^2}, & \mbox{if } c=0, \\
\frac{x}{1-qx}+\frac{x+x^2}{1+qx^2}, & \mbox{if } c\ne 0 \text{\ is\ square},\\
\frac{x}{1-qx}+\frac{-x+x^2}{1+qx^2}, & \mbox{if } c \text{\ is\ non-square}.
\end{cases}
\end{align*}
\end{thm}
Evidently, Theorem \ref{thm1.2} extends Myerson's theorem \cite{[My1]}.

Let $M_n^{(e)}(y)$ be the number of zeros of
\begin{equation*}
x_1^e+\dots+x_{n-1}^e+yx_n^e=0
\end{equation*}
with $y\in\mathbb{F}_q^*$ not being the $e$-th power of any
element in $\mathbb{F}_q$. Gauss \cite{[Gauss]} showed
that if $q=p\equiv1\pmod3$ and $y$ is non-cubic, then $M_3^{(3)}(y)=p^2+\frac{1}{2}(p-1)(-c+9d)$,
where $c$ and $d$ are uniquely determined by
$4p=c^2+27d^2,~c\equiv 1\pmod 3$ except for the sign of $d$.
In 1978, Chowla, Cowles and Cowles \cite{[CCC78]} determined
the sign of $d$ for the case of $2$ being a non-cubic element of
${\mathbb F}_p$. Hong and Zhu \cite{[HZ]} determined the
sign of $d$ when $2$ is cubic in ${\mathbb F}_p$. Furthermore,
they showed that the generating function
$\sum_{n=1}^{\infty} M_{n+1}^{(3)}(y)x^{n}$ is a rational
function for any $y\in\mathbb F_q^*$ with $y$ being
non-cubic over ${\mathbb F}_q$ and also presented
its explicit expression.

For brevity, we let $M_{n}(y):=M_n^{(4)}(y)$ in the
remaining part of this paper. With the help of Theorem
\ref{thm1.2}, we prove that the generating function
$\sum_{n=1}^{\infty}M_{n+1}(y)x^n$ is rational
as the following second main result of this paper shows.

\begin{thm}\label{thm1.3}
Let $\mathbb{F}_q$ be the finite field of $q=p^m$ elements
with $p$ being an odd prime. Let
$y\in\mathbb{F}_q^*=\langle g\rangle$ be non-quartic.
Then each of the following is true.

{\rm (i).} If $q\equiv 1\pmod 4$, then
\begin{equation*}
\sum_{n=1}^{\infty}M_{n+1}(y)x^n=
\begin{cases}
\frac{qx}{1-qx}+\frac{(q-1)
(3x^2+B_1(y,x))}{1-6qx^2+8qsx^3+(q^2-4qs^2)x^4},
& \mbox{if } q\equiv 1\pmod 8, \\
\frac{qx}{1-qx}+\frac{(q-1)
(-x^2+B_2(-y,x))}{1+2qx^2+8qsx^3+(9q^2-4qs^2)x^4},
& \mbox{if } q\equiv 5\pmod 8,
\end{cases}
\end{equation*}
where $s$, $t$, $B_1(y,x)$ and $B_2(-y,x)$ are
given as in Theorem \ref{thm1.2}.

{\rm (ii).} If $q\equiv 3\pmod 4$, then
$$\sum_{n=1}^{\infty}M_{n+1}(y)x^n=\frac{qx}{1-qx}+\frac{(q-1)x}{1+qx^2}.$$
\end{thm}

This paper is organized as follows. In Section 2, we present several
lemmas on dimension $n$ cyclotomic numbers of order $k$. Section 3 is
devoted to computing the values of $N_n(c)$ for $1\le n\le 4$.
In Sections 4 and 5, we provide the proofs of Theorems \ref{thm1.2}
and \ref{thm1.3}, respectively.

\section{Dimension $n$ cyclotomic numbers of order $k$}
Throughout, we always let $q=p^{m}=kf+1$
with $k$ and $f$ being positive integers. For any
integer $r\ge 1$, we define $\langle r\rangle:=\{0,...,r-1\}$.
We begin with a well-known result on the solutions
of linear congruence.

\begin{lem}\label{lem2.1} \cite{[Mc]}
Let $a_1,\dots,a_n$ be nonzero integers
and let $r>1$ be a positive integer.
Suppose $d=\gcd(a_1,\ \dots,\ a_n,\ r)$. Then the congruence
$$a_1x_1 + \dots + a_nx_n \equiv b\ (\bmod \ r)$$
has a solution if and only if $d|b$. If $d|b$, then
the above congruence has exactly $dr^{n-1}$ solutions
$(x_1, ..., x_n)$ modulo $r$.
\end{lem}

In particular, we have the following result on
restricted solutions of the linear congruence.

\begin{lem}\label{cor2.1}
Let $a_1,\dots, a_n$ and $b$ be nonzero integers and let
$d=\gcd(a_1, \dots, a_n, kf)$. If $d|b$, then the number
of the restricted solutions $(x_1, ..., x_n)$ with
$0\leq x_1,\dots,x_n\leq \frac{kf}{d}-1$ of the
linear congruence
\begin{equation}\label{eq2.1}
a_1x_1 + \dots + a_nx_n \equiv b \pmod {kf}
\end{equation}
is equal to $(\frac{kf}{d})^{n-1}$.
\end{lem}

\begin{proof}
The condition that $d|b$ is certainly necessary for
the congruence to have a solution.

On the other hand, suppose that $d|b$. The congruence
(\ref{eq2.1}) is equivalent to
\begin{equation}\label{eq2.2}
\frac{a_1}{d}x_1 + \dots + \frac{a_n}{d}x_n
\equiv \frac{b}{d} \pmod  {\frac{kf}{d}}.
\end{equation}
The congruence (\ref{eq2.2}) has exactly
$(\frac{kf}{d})^{n-1}$ solutions from Lemma
\ref{lem2.1}.
\end{proof}

The next lemma provides a basic fact in the cyclotomic
theory in $\mathbb{F}_q$.

\begin{lem}\label{lem2.2}\cite{[St]}
Let $\mathbb{F}_{q}^{*}=\langle g\rangle$ and let
$g^{ku_1+i}=g^{ku_2+j}$ with $0\leq u_1,\ u_2\leq f-1$
and $0\leq i,\ j\leq k-1$. Then $u_1=u_2$ and $i=j$.
\end{lem}

\begin{defn}\label{defn2.1}
Let $\mathbb{F}_{q}^{*}=\langle g\rangle$ and
let $i$ and $j$ be two integers. We define the
{\it cyclotomic number of order $k$ with respect
to $g$}, denoted by  $(i, j)_k$, to be the number
of pairs $(u_1, u_2)\in \langle f\rangle^2$
satisfying that $1+g^{ku_1+i}=g^{ku_2+j}$, where
$1=g^0$ is the multiplicative identity of $\mathbb{F}_q$.
\end{defn}

\begin{rem}\label{rem2.1}
For any $i\in\langle k\rangle$, we define the
{\it cyclotomic classes} $C_i$ with respect to
the generator $g$ of $\mathbb{F}_q^*$
as follows:
\begin{equation*}
C_i:=\{g^{i+ku}: u=0,\ 1,\ \dots,\ f-1\}.
\end{equation*}
By Lemma \ref{lem2.2}, the $k$ cyclotomic classes
$C_0, C_1, ..., C_{k-1}$ are pairwise disjoint,
and their union equals $\mathbb{F}_q^*$.
Evidently, $C_{i+vk}=C_i$ for any integer $v$.
The cyclotomic number $(i,\ j)_{k}$ is equal to
the number of solutions $(x_i, x_j)\in C_i\times C_j$
of the equation $x_i +1=x_j$.
\end{rem}

We now exhibit some properties of cyclotomic number as follows.

\begin{lem}\label{lem2.3}\cite{[St]}
Each of the following is true.
\begin{enumerate}[{\rm (1).}]
\item For any integers $m_1$ and $m_2$,
$(i+m_1k, j+m_2k)_{k}=(i, j)_{k}$.
\item $(i, j)_{k}=(-i, j-i)_{k}$.
\item $(i, j)_{k}=
\begin{cases}
(j, i)_{k}, & \mbox{if }\ f\ is\ even, \\
(j+\frac{k}{2}, i+\frac{k}{2})_{k}, & \mbox{if }\ f\ is\ odd.
\end{cases}$
\end{enumerate}
\end{lem}

The next lemma provides explicit formulas of
cyclotomic numbers of order $4$ for $\mathbb{F}_q$.

\begin{lem}\label{lem2.4} \cite{[KR],[My2]}
Let $q=p^{m}=4f+1$ with $p$ being an odd prime and
let $g$ be a generator of $\mathbb{F}_{q}^{*}$.
If $p\equiv 3\ \pmod 4$, then let
$s=(-p)^{\frac{m}{2}}$ and $t=0$. If $p\equiv 1\ \pmod 4$,
then $s$ and $t$ are uniquely determined by
\begin{align}\label{st}
q=s^{2}+4t^{2} \ {\it with} \ p\nmid s,
s\equiv 1 \pmod 4 \ {\it and} \
2t\equiv sg^{\frac{3(q-1)}{4}} \pmod p.
\end{align}
Then all the cyclotomic numbers of order $4$ for $\mathbb{F}_q$,
corresponding to $g$, are determined unambiguously
as follows:

For $f$ even, one has
\begin{align*}
&A=(0, 0)_{4}=\frac{1}{16}(q-11-6s),\\
&B=(0, 1)_{4}=(1, 0)_{4}=(3, 3)_{4}=\frac{1}{16}(q-3+2s+8t),\\
&C=(0, 2)_{4}=(2, 0)_{4}=(2, 2)_{4}=\frac{1}{16}(q-3+2s),\\
&D=(0, 3)_{4}=(3, 0)_{4}=(1, 1)_{4}=\frac{1}{16}(q-3+2s-8t),\\
&E=(1, 2)_{4}=(2, 1)_{4}=(1, 3)_{4}=(3, 1)_{4}=(2, 3)_{4}=(3, 2)_{4}=\frac{1}{16}(q+1-2s),
\end{align*}
and for $f$ odd, we have
\begin{align*}
&A=(0, 0)_{4}=(2, 0)_{4}=(2, 2)_{4}=\frac{1}{16}(q-7+2s),\\
&B=(0, 1)_{4}=(1, 3)_{4}=(3, 2)_{4}=\frac{1}{16}(q+1+2s-8t),\\
&C=(0, 2)_{4}=\frac{1}{16}(q+1-6s),\\
&D=(0, 3)_{4}=(1, 2)_{4}=(3, 1)_{4}=\frac{1}{16}(q+1+2s+8t),\\
&E=(1, 0)_{4}=(1, 1)_{4}=(2, 1)_{4}=(2, 3)_{4}=(3, 0)_{4}=(3, 3)_{4}=\frac{1}{16}(q-3-2s).
\end{align*}
\end{lem}

Let us now define the generalized cyclotomic number.

\begin{defn}\label{defn2.2}
Let $q=p^{m}$ and $\mathbb{F}_{q}^{*}=\langle g\rangle$.
Let $k, f$ and $n$ be positive integers such that $q-1=kf$.
Let $i_1, ..., i_{n-1}$ and $i_n$ be $n$ integers. Then we define
the {\it dimension $n$ cyclotomic number of order $k$
 with respect to $g$}, denoted by $[i_1,\dots,i_n]_{k}$,
to be the number of $n$-tuples
$(u_1,\dots,u_n)\in\langle f \rangle^n$ satisfying
\begin{equation}\label{eq2.3}
g^{ku_1+i_1}+\cdots +g^{ku_n+i_n}=1.
\end{equation}
\end{defn}

\begin{rem}\label{rem2.2}
Similar to cyclotomic number, the cyclotomic number
\noindent $[i_1, ..., i_n]_{k}$ can also be
viewed as the number of solutions
of the equation
\begin{equation*}
x_1 +\dots +x_n =1\ \ \ (x_j\in C_{i_j},\ 1\leq j\leq n).
\end{equation*}
For any integers $m_1,\ \dots,\ m_n$ and $v$,
since $C_{i+vk}=C_i$, one can deduce that $$[i_1+m_1k,\dots,i_n+m_nk]_{k}=[i_1,\dots,i_n]_{k}.$$
\end{rem}

If $n=1$, then (\ref{eq2.3}) becomes
$g^{ku_1+i_1}=1$ with $0\leq u_1 \leq f-1$,
which is equivalent to the congruence
\begin{equation}\label{eq2.4}
ku_1 \equiv -i_1 \ (\bmod \ q-1)\  \text{with} \ 0\leq u_1 \leq f-1.
\end{equation}
By Corollary \ref{cor2.1}, (\ref{eq2.4}) has
a solution if $i_1\equiv 0\ \pmod k$
and no solution if $i_1\not\equiv 0\ \pmod k$.
Thus
\begin{equation}\label{eq2.5}
[i_1]_{k}=
\begin{cases}
1, & \mbox{if}\ i_1\equiv 0\ \pmod k,\\
0, & \mbox{if}\ i_1\not\equiv 0\ \pmod k.
\end{cases}
\end{equation}
In particular, we have the following result.

\begin{lem}\label{exm2.1}
We have
$[0]_{4}=1$ and $[1]_{4}=[2]_{4}=[3]_{4}=0.$
\end{lem}

The next lemma expresses the dimension 2 cyclotomic number of order $k$
in terms of the cyclotomic number of order $k$.

\begin{lem}\label{lem2.5} We have
$[i_1, i_2]_{k}=(i_2-i_1,\ -i_1)_{k}.$
\end{lem}
\begin{proof}
Letting $n=2$ in (2.4) and noticing that
$g^{\frac{q-1}{2}}=g^{\frac{kf}{2}}=-1$,
one arrives at
\begin{equation*}
1+g^{ku_1+i_1+\frac{kf}{2}}=g^{ku_2+i_2},\ \ \ 0\leq u_1, u_2\leq f-1,
\end{equation*}
which, by the definition of cyclotomic number, has
$(i_1+\frac{kf}{2},\ i_2)_{k}$ solutions $(u_1, u_2)\in\langle f\rangle^2$.
Thus applying Lemma \ref{lem2.3} gives us that
\begin{align}\label{eq2.6}
[i_1,\ i_2]_{k}&=(i_1+\frac{kf}{2},\ i_2)_{k}\notag\\
&=(-i_1-\frac{kf}{2},\ i_2-i_1-\frac{kf}{2})_{k}\notag\\
&=\begin{cases}\notag
(i_2-i_1-\frac{kf}{2},\ -i_1-\frac{kf}{2})_{k}, & \mbox{if }\ f\ {\rm is \ even,}\\
(i_2-i_1-\frac{kf}{2}+\frac{k}{2},\ -i_1-\frac{kf}{2}+\frac{k}{2})_{k},
& \mbox{if }\ f\ {\rm is \ odd}
\end{cases}\\
&=\begin{cases}\notag
(i_2-i_1,\ -i_1)_{k}, & \mbox{if }\ f\ {\rm is \ even,}\\
(i_2-i_1,\ -i_1)_{k}, & \mbox{if }\ f\ {\rm is \ odd}
\end{cases}\\
&=(i_2-i_1,\ -i_1)_{k}\notag
\end{align}
as expected. The proof of Lemma \ref{lem2.5} is complete.
\end{proof}

From Lemma \ref{lem2.3}, Lemma \ref{lem2.4} and Lemma
\ref{lem2.5}, the following result follows immediately.
\begin{lem}\label{exm2.2}
One has
\begin{align*}
&[0,0]_{4}=(0, 0)_{4}=
\begin{cases}
\frac{1}{16}(q-6s-11), & \mbox{if }  q\equiv 1 \pmod 8, \\
\frac{1}{16}(q+2s-7), & \mbox{if }  q\equiv 5 \pmod 8,
\end{cases}
\\
&[1, 1]_{4}=(0, -1)_{4}=(0, 3)_{4}=
\begin{cases}
\frac{1}{16}(q+2s-8t-3), & \mbox{if }  q\equiv 1 \pmod 8, \\
\frac{1}{16}(q+2s+8t+1), & \mbox{if }  q\equiv 5 \pmod 8,
\end{cases}
\\
&[2,2]_{4}=(0, -2)_{4}=(0, 2)_{4}=
\begin{cases}
\frac{1}{16}(q+2s-3), & \mbox{if }  q\equiv 1 \pmod 8, \\
\frac{1}{16}(q-6s+1), & \mbox{if }  q\equiv 5 \pmod 8,
\end{cases}
\\
&[3, 3]_{4}=(0, -3)_{4}=(0, 1)_{4}=
\begin{cases}
\frac{1}{16}(q+2s+8t-3), & \mbox{if }  q\equiv 1 \pmod 8,
\\
\frac{1}{16}(q+2s-8t+1), & \mbox{if }  q\equiv 5 \pmod 8.
\end{cases}
\end{align*}
where $s$ and $t$ are given as in (\ref{st}).
\end{lem}

In the next lemma, we determine
$[i_1, i_2, i_3]_{k}$ in terms of cyclotomic numbers
of order $k$.
\begin{lem}\label{lem2.6} Each of the following is true.

{\rm (i)}. For any integers $i_1, i_2$ and $i_3$, we have
\begin{equation*}
[i_1, i_2, i_3]_{k}=\alpha+\sum_{v=0}^{k-1}
(v-i_3, -i_3)_{k}(i_2-i_1, v-i_1)_{k},
\end{equation*}
where
\begin{equation*}
\alpha=
\begin{cases}
f, & \mbox{if} \ \  i_1\equiv i_2+\frac{kf}{2} \pmod k
\  {\it and}\   i_3\equiv 0 \pmod k, \\
0, & \mbox{otherwise}.
\end{cases}
\end{equation*}

{\rm (ii)}. For any integer $i$, we have
\begin{equation*}
[i, i, i]_{k}=\beta+\sum_{v=0}^{k-1}(v, -i)_{k} (0, v)_{k},
\end{equation*}
with
\begin{equation*}
\beta:=
\begin{cases}
f, & \mbox{if }\   i\equiv 0 \pmod k \  and \  f\   is \  even, \\
0, & \mbox{otherwise}.
\end{cases}
\end{equation*}
\end{lem}

\begin{proof}
(i). By the definition of dimension 3 cyclotomic number of order $k$, one has
\begin{align}\label{eq2.8'}
[i_1, i_2, i_3]_{k}&=\sum_{\substack{(u_1, u_2, u_3)\in\langle f\rangle^3\\
g^{ku_1+i_1}+g^{ku_2+i_2}+g^{ku_3+i_3}=1}} 1\notag\\
&=\sum_{\substack{(u_1, u_2, u_3)\in\langle f\rangle^3 \\
g^{ku_1+i_1}+g^{ku_2+i_2}=0\\g^{ku_3+i_3}=1}} 1+\sum_{u=0}^{f-1}
\sum_{v=0}^{k-1}\sum_{\substack{(u_1, u_2, u_3)\in\langle f\rangle^3 \\
g^{ku_1+i_1}+g^{ku_2+i_2}=g^{ku+v}\\
g^{ku+v}+g^{ku_3+i_3}=1}} 1 \notag\\
&:=S_1+S_2.
\end{align}

First of all, we have
\begin{equation*}
S_1=\Bigg(\sum_{\substack{u_3=0\\
g^{ku_3+i_3}=1}}^{f-1}1\Bigg)\Bigg(\sum_{\substack{u_1, u_2=0\\
g^{ku_1+i_1}+g^{ku_2+i_2}=0}}^{f-1}1\Bigg)
=[i_3]_{k}\sum_{\substack{(u_1, u_2)\in\langle f\rangle^2\\
g^{ku_1+i_1}+g^{ku_2+i_2}=0}} 1.
\end{equation*}

Since $g^{\frac{q-1}{2}}=g^{\frac{kf}{2}}=-1$, the equation
$g^{ku_1+i_1}+g^{ku_2+i_2}=0$ can be rewritten in the following
form
\begin{equation*}
g^{k(u_2-u_1)}=-g^{i_1-i_2}=g^{i_1-i_2+\frac{kf}{2}},
\end{equation*}
which is equivalent to the congruence
\begin{equation*}
-ku_1+ku_2\equiv i_1-i_2+\frac{kf}{2} \ (\bmod\ kf).
\end{equation*}

By Lemma \ref{cor2.1}, this congruence has $f$
solutions $(u_1, u_2)\in\langle f\rangle^2$
if $i_1-i_2+\frac{kf}{2}\equiv 0\ \pmod k$,
and $0$ solutions $(u_1, u_2)$, if $k$ does not
divide $i_1-i_2+\frac{kf}{2}$.
Hence, by (\ref{eq2.5}), one gets that
\begin{equation}\label{eq2.8}
S_1=
\begin{cases}
f, & \mbox{if}\ i_1-i_2+\frac{kf}{2}\equiv i_3\equiv 0\ \pmod k, \\
0, & \mbox{otherwise}.
\end{cases}
\end{equation}

Consequently, we calculate $S_2$. By Lemma \ref{lem2.5}
and Lemma \ref{lem2.3} (2), one derives that
\begin{align}\label{eq2.9}
S_2&=\sum_{v=0}^{k-1}\sum_{u=0}^{f-1}\sum_{\substack
{(u_1, u_2, u_3)\in\langle f\rangle^3 \\g^{ku_1+i_1}+g^{ku_2+i_2}=g^{ku+v}
\\g^{ku+v}+g^{ku_3+i_3}=1}} 1\notag\\
&=\sum_{v=0}^{k-1}\sum_{u=0}^{f-1}\sum_{\substack{u_3\in\langle f\rangle
\\g^{ku+v}+g^{ku_3+i_3}=1}}
\sum_{\substack{(u_1, u_2)\in\langle f\rangle^2 \\g^{ku_1+i_1}+g^{ku_2+i_2}=g^{ku+v}}} 1
\notag\\
&=\sum_{v=0}^{k-1}\sum_{u=0}^{f-1}\Big(\sum_{\substack{u_3\in\langle f\rangle
\\g^{ku+v}+g^{ku_3+i_3}=1}} 1\Big)\Big(\sum_{\substack{(u_1, u_2)\in\langle f\rangle^2 \\g^{ku_1+i_1}+g^{ku_2+i_2}=g^{ku+v}}} 1\Big)\notag
\\
&=\sum_{v=0}^{k-1}\sum_{u=0}^{f-1}\Big(\sum_{\substack{u_3\in\langle f\rangle
\\g^{ku+v}+g^{ku_3+i_3}=1}} 1\Big)\Big(\sum_{\substack{(u_1, u_2)\in\langle f\rangle^2 \\g^{k(u_1-u)+i_1-v}+g^{k(u_2-u)+i_2-v}=1}} 1\Big)\notag
\\
&=\sum_{v=0}^{k-1}\sum_{u=0}^{f-1}\Big(\sum_{\substack{u_3\in\langle f\rangle
\\g^{ku+v}+g^{ku_3+i_3}=1}} 1\Big)\Big(\sum_{\substack{(u_1, u_2)\in\langle f\rangle^2 \\g^{ku_1+i_1-v}+g^{ku_2+i_2-v}=1}} 1\Big)\notag
\\
&=\sum_{v=0}^{k-1}\Big(\sum_{\substack{(u_1, u_2)\in\langle f\rangle^2 \\g^{ku_1+i_1-v}+g^{ku_2+i_2-v}=1}} 1\Big)\sum_{u=0}^{f-1}\Big(\sum_{\substack{u_3\in\langle f\rangle
\\g^{ku+v}+g^{ku_3+i_3}=1}} 1\Big)\notag
\\
&=\sum_{v=0}^{k-1}\Big(\sum_{\substack{(u_1, u_2)\in\langle f\rangle^2 \\g^{ku_1+i_1-v}+g^{ku_2+i_2-v}=1}} 1\Big)
\Big(\sum_{\substack{(u,u_3)\in\langle f\rangle^2
\\g^{ku+v}+g^{ku_3+i_3}=1}} 1\Big)\notag
\\
&=\sum_{v=0}^{k-1}[i_1-v, i_2-v]_{k} [v, i_3]_{k}\notag\\
&=\sum_{v=0}^{k-1}(i_2-i_1, v-i_1)_{k}(i_3-v, -v)_{k}\notag\\
&=\sum_{v=0}^{k-1}(i_2-i_1, v-i_1)_{k} (v-i_3, -i_3)_{k}.
\end{align}
Putting (\ref{eq2.8}) and (\ref{eq2.9}) into (\ref{eq2.8'})
gives us the desired result. Part (i) is proved.

(ii). If $i\equiv 0\ \pmod k$ and $f$ is even,
then by part (i), one has
\begin{equation*}
[0, 0, 0]_{k}=f+\sum_{v=0}^{k-1}(v, 0)_{k} (0, v)_{k}.
\end{equation*}

If $i\not\equiv 0\ \pmod k$ or $f$ is odd,
then part (i) tells us that
\begin{equation*}
[i, i, i]_{k}=\sum_{v=0}^{k-1}(v-i, -i)_{k} (0, v-i)_{k}.
\end{equation*}

Since $v-i$ runs over a complete residue system modulo
$k$ as $v$ runs from $0$ to $k-1$, by Lemma \ref{lem2.3} (1),
one can deduce that
\begin{equation*}
[i, i, i]_{k}=\sum_{v=0}^{k-1}(v, -i)_{k} (0, v)_{k}
\end{equation*}
as desired. Part (ii) is proved.

This completes the proof of Lemma \ref{lem2.6}.
\end{proof}

Picking $k=4$ in Lemma \ref{lem2.6}, it follows from
Lemma \ref{lem2.3} (1) and Lemma \ref{lem2.4} that the
following result is true.

\begin{lem}\label{exm2.3}
We have
\begin{align*}
&[0, 0, 0]_{4}=
\begin{cases}
\frac{1}{64}(q^2+14q+4s^2+24s+21), & \mbox{if }  q\equiv 1 \pmod 8, \\
\frac{1}{64}(q^2-6q-4s^2+9), & \mbox{if }  q\equiv 5 \pmod 8,
\end{cases}
\\
&[1, 1, 1]_{4}=
\begin{cases}
\frac{1}{64}(q^2-10q+8st+24t+9), & \mbox{if }  q\equiv 1 \pmod 8, \\
\frac{1}{64}(q^2+2q-8st-24t-3), & \mbox{if }  q\equiv 5 \pmod 8,
\end{cases}
\\
&[2, 2, 2]_{4}=
\begin{cases}
\frac{1}{64}(q^2-6q-4s^2+9), & \mbox{if }  q\equiv 1 \pmod 8, \\
\frac{1}{64}(q^2-6q-16t^2+24s-3), & \mbox{if }  q\equiv 5 \pmod 8,
\end{cases}
\end{align*}
and
\begin{align*}
&[3, 3, 3]_{4}=
\begin{cases}
\frac{1}{64}(q^2-10q-8st-24t+9), & \mbox{if }  q\equiv 1 \pmod 8, \\
\frac{1}{64}(q^2+2q+8st+24t-3), & \mbox{if }  q\equiv 5 \pmod 8,
\end{cases}
\end{align*}
where $s$ and $t$ are as given in (\ref{st}).
\end{lem}

\begin{proof}
 For any integer $i\in \{0,1,2,3\}$, we have
\begin{equation*}
[i, i, i]_{4}=\beta+\sum_{v=0}^{3}(v, -i)_{4} (0, v)_{4},
\end{equation*}
with
\begin{equation*}
\beta:=
\begin{cases}
f, & \mbox{if }\   i\equiv 0 \pmod 4 \  \mbox{and} \
f\  \mbox{is even}, \\
0, & \mbox{otherwise}.
\end{cases}
\end{equation*}

Let $q\equiv 1\pmod 8$ and $i=1$. Then by Lemmas 
\ref{lem2.4} and \ref{lem2.6} (ii), we have
$\beta=0$ and
\begin{align*}
[1, 1, 1]_{4}=&(0, 3)_{4} (0, 0)_{4}+(1, 3)_{4} (0, 1)_{4}
+(2, 3)_{4} (0, 2)_{4}+(3, 3)_{4} (0, 3)_{4}\\
=&\frac{1}{256}\big((q-3+2s-8t)(q-11-6s)+(q+1-2s)(q-3+2s+8t)\\
&+(q+1-2s)(q-3+2s)+(q-3+2s+8t)(q-3+2s-8t)\big)\\
=&\frac{1}{64}(q^2-6q-4s^2-16t^2+8st+24t+9)\\
=&\frac{1}{64}(q^2-10q+8st+24t+9)
\end{align*}
since $q=s^2+4t^2$.

For the other cases, using Lemmas \ref{lem2.4} and \ref{lem2.6}
(ii) and by computations, we can get the desired results.
The proof of Lemma \ref{exm2.4}
is complete.
\end{proof}

In the next lemma, we supply the relation between
$[i_1, i_2, i_3, i_4]_{k}$ and cyclotomic numbers.

\begin{lem}\label{lem2.7} Each of the following is true:

{\rm (i)}. For any integers $i_1, i_2,i_3$ and $i_4$, we have
\begin{equation*}
[i_1, i_2, i_3, i_4]_{k}=\gamma+\sum_{(v_1, v_2)\in\langle k
\rangle^2}(v_2-v_1, -v_1)_{k} (i_2-i_1, v_1-i_1)_{k}
(i_4-i_3, v_2-i_3)_{k},
\end{equation*}
where
\begin{small}
\begin{equation*}
\gamma:=
\begin{cases}
(i_2-i_1, -i_1)_{k} f+(i_4-i_3, -i_3)_{k} f, & \mbox{if }i_2-i_1
\equiv\frac{kf}{2}\pmod k \ \mbox{and}\  i_4-i_3
\equiv\frac{kf}{2} \pmod k, \\
(i_2-i_1, -i_1)_{k} f, & \mbox{if }i_2-i_1\not\equiv\frac{kf}{2}
\pmod k \ \mbox{and}\  i_4-i_3\equiv\frac{kf}{2} \pmod k, \\
(i_4-i_3, -i_3)_{k} f, &\mbox{if }i_2-i_1\equiv\frac{kf}{2}
\pmod k \ \mbox{and}\  i_4-i_3\not\equiv\frac{kf}{2} \pmod k, \\
0, & \mbox{otherwise}.
\end{cases}
\end{equation*}
\end{small}

{\rm (ii)}. For any integer $i$, we have
\begin{equation*}
[i, i, i, i]_{k}=2\theta +\sum_{(v_1, v_2)\in\langle k
\rangle^2}
(v_2-v_1, -v_1)_{k} (0, v_1-i)_{k} (0, v_2-i)_{k},
\end{equation*}
where
\begin{equation*}
\theta=
\begin{cases}
(0, -i)_{k} f , & \mbox{if } \ f\  is \ even, \\
0, & \mbox{if }\  f \ is\  odd.
\end{cases}
\end{equation*}
\end{lem}

\begin{proof}
(i). By the definition of dimension 4 cyclotomic number of order $k$, one has
\begin{align}\label{eq2.10}
&[i_1, i_2, i_3, i_4]_{k}\notag\\
=&\sum_{\substack{(u_1, u_2, u_3, u_4)\in \langle f\rangle^4
\\g^{ku_1+i_1} + g^{ku_2+i_2} + g^{ku_3+i_3} + g^{ku_4+i_4} = 1}} 1\notag\\
\notag\\
=&\sum_{\substack{(u_1, u_2, u_3, u_4)\in \langle f\rangle^4
\\g^{ku_1+i_1} + g^{ku_2+i_2} = 1\\g^{ku_3+i_3} + g^{ku_4+i_4} = 0}} 1  \
+\sum_{\substack{(u_1, u_2, u_3, u_4)\in \langle f\rangle^4
\\g^{ku_1+i_1} + g^{ku_2+i_2} = 0\\g^{ku_3+i_3} + g^{ku_4+i_4} = 1}} 1
+\sum_{(v_1,v_2)\in \langle k\rangle^2}\sum_{(s_1,s_2)\in
\langle f\rangle^2}\sum_{\substack{(u_1, u_2, u_3, u_4)\in \langle f\rangle^4
\\g^{ku_1+i_1}+g^{ku_2+i_2}=g^{ks_1+v_1}\\g^{ku_3+i_3}+g^{ku_4+i_4}
= g^{ks_2+v_2}\\g^{ks_1+v_1} + g^{ks_2+v_2} = 1}} 1\notag\\
\notag\\
=&:T_1+T_2+T_3.
\end{align}

First of all, applying Lemma \ref{lem2.5} we have
\begin{align*}
T_1=& \sum_{\substack{(u_1, u_2, u_3, u_4)\in \langle f\rangle^4
\\g^{ku_1+i_1} + g^{ku_2+i_2} = 1\\g^{ku_3+i_3} + g^{ku_4+i_4} = 0}} 1
=\Big(\sum_{\substack{(u_1, u_2)\in \langle f\rangle^2\\
g^{ku_1+i_1}+g^{ku_2+i_2}=1}} 1\Big)\Big(\sum_{\substack{(u_3, u_4)\in \langle f\rangle^2\\
g^{ku_3+i_3}+g^{ku_4+i_4} = 0}} 1\Big)\\
=& [i_1, i_2]_{k}\sum_{\substack{(u_3, u_4)\in \langle f\rangle^2\\
g^{ku_3+i_3}+g^{ku_4+i_4} = 0}} 1\\
=& (i_2-i_1, -i_1)_{k}\sum_{\substack{(u_3, u_4)
\in\langle f\rangle^2\\g^{ku_3+i_3} + g^{ku_4+i_4} = 0}} 1.
\end{align*}
Since $g^{\frac{kf}{2}}=-1$, $g^{ku_3+i_3}+g^{ku_4+i_4}=0$ is equivalent to
\begin{equation*}
g^{ku_4+i_4}=-g^{ku_3+i_3}=g^{ku_3+i_3+\frac{kf}{2}},
\end{equation*}
which is true if and only if
\begin{equation*}
k(u_4-u_3)\equiv (i_3-i_4)+\frac{kf}{2} \pmod {q-1}.
\end{equation*}
By Lemma \ref{cor2.1}, this congruence has $f$ or $0$
solutions $(u_3, u_4)$ according to $k$ dividing
$(i_3-i_4)+\frac{kf}{2}$ or not. Thus
\begin{equation}\label{eq2.11}
\sum_{\substack{(u_3, u_4)
\in\langle f\rangle^2\\g^{ku_3+i_3} + g^{ku_4+i_4} = 0}} 1
=\begin{cases}
f, & \mbox{if } i_3-i_4+\frac{kf}{2}\equiv 0 \pmod k, \\
0, & \mbox{otherwise},
\end{cases}
\end{equation}
and so
\begin{equation}\label{eq2.11}
T_1=
\begin{cases}
(i_2-i_1, -i_1)_{k} f, & \mbox{if } i_3-i_4+\frac{kf}{2}\equiv 0 \pmod k, \\
0, & \mbox{otherwise}.
\end{cases}
\end{equation}

By the symmetry, one can exchange $i_1$ and $i_3$, and
exchange $i_2$ and $i_4$ in (\ref{eq2.11}), and arrives at
\begin{equation}\label{eq2.12}
T_2=
\begin{cases}
(i_4-i_3, -i_3)_{k} f, & \mbox{if } i_1-i_2+\frac{kf}{2}\equiv 0 \pmod k, \\
0, & \mbox{otherwise}.
\end{cases}
\end{equation}

Now we turn to the evaluation of $T_3$. We can compute that
\begin{align}\label{eq2.13}
T_3&=\sum_{(v_1,v_2)\in \langle k\rangle^2}\sum_{(s_1,s_2)\in \langle f\rangle^2}\sum_{\substack{(u_1, u_2, u_3, u_4)\in \langle f\rangle^4
\\g^{ku_1+i_1} + g^{ku_2+i_2} = g^{ks_1+v_1}\\g^{ku_3+i_3} + g^{ku_4+i_4} = g^{ks_2+v_2}\\g^{ks_1+v_1} + g^{ks_2+v_2} = 1}} 1\notag\\
&=\sum_{(v_1, v_2)\in \langle k\rangle^2}\sum_{\substack{(s_1, s_2)\in \langle f\rangle^2\\g^{ks_1+v_1} + g^{ks_2+v_2} = 1}}  \
\sum_{\substack{(u_1, u_2, u_3, u_4)\in \langle f\rangle^4\\g^{ku_1+i_1} + g^{ku_2+i_2} = g^{ks_1+v_1}\\
g^{ku_3+i_3} + g^{ku_4+i_4} = g^{ks_2+v_2}}}  1\notag\\
\notag\\
&=\sum_{(v_1, v_2)\in \langle k\rangle^2}\sum_{\substack{(s_1, s_2)\in \langle f\rangle^2\\g^{ks_1+v_1} + g^{ks_2+v_2} = 1}}  \
\sum_{\substack{(u_1, u_2, u_3, u_4)\in \langle f\rangle^4
\\g^{k(u_1-s_1)+(i_1-v_1)} + g^{k(u_2-s_1)+(i_2-v_1)} = 1\\
g^{k(u_3-s_2)+(i_3-v_2)} + g^{k(u_4-s_2)+(i_4-v_2)} = 1}} 1\notag\\
\notag\\
&=\sum_{(v_1, v_2)\in \langle k\rangle^2}\sum_{\substack{(s_1, s_2)\in \langle f\rangle^2\\g^{ks_1+v_1} + g^{ks_2+v_2} = 1}}  \
\sum_{\substack{(u_1, u_2, u_3, u_4)\in \langle f\rangle^4
\\g^{ku_1+(i_1-v_1)} + g^{ku_2+(i_2-v_1)} = 1\\
g^{ku_3+(i_3-v_2)} + g^{ku_4+(i_4-v_2)} = 1}} 1\notag\\
\notag\\
&=\sum_{(v_1, v_2)\in \langle k\rangle^2}
\sum_{\substack{(s_1, s_2)\in \langle f\rangle^2\\g^{ks_1+v_1} + g^{ks_2+v_2} = 1}}
\Big(\sum_{\substack{(u_1, u_2)\in \langle f\rangle^2
\\g^{ku_1+(i_1-v_1)} + g^{ku_2+(i_2-v_1)}=1}} 1
\Big)\Big(\sum_{\substack{(u_3, u_4)\in \langle f\rangle^2
\\g^{ku_3+(i_3-v_2)} + g^{ku_4+(i_4-v_2)}=1}} 1\Big)\notag\\
\notag\\
&=\sum_{(v_1, v_2)\in \langle k\rangle^2}
\Big(\sum_{\substack{(u_1, u_2)\in \langle f\rangle^2
\\g^{ku_1+(i_1-v_1)} + g^{ku_2+(i_2-v_1)}=1}} 1
\Big)\Big(\sum_{\substack{(u_3, u_4)\in \langle f\rangle^2
\\g^{ku_3+(i_3-v_2)} + g^{ku_4+(i_4-v_2)}=1}} 1\Big)
\sum_{\substack{(s_1, s_2)\in \langle f\rangle^2\\
g^{ks_1+v_1}+g^{ks_2+v_2} = 1}} 1\notag\\
&=\sum_{(v_1, v_2)\in \langle k\rangle^2 }
[i_1-v_1, i_2-v_1]_{k} [i_3-v_2, i_4-v_2]_{k}[v_1, v_2]_{k} \notag\\
&=\sum_{(v_1, v_2)\in \langle k\rangle^2 }(i_2-i_1, v_1-i_1)_{k}
(i_4-i_3, v_2-i_3)_{k}(v_2-v_1, -v_1)_{k},
\end{align}
where Lemma \ref{lem2.5} is applied in the last step.
Putting (\ref{eq2.11}) to (\ref{eq2.13}) into (\ref{eq2.10})
gives us the desired result.
This finishes the proof of part (i).

(ii). Letting $i_1=i_2=i_3=i_4=i$ in part (i), one obtains that
\begin{equation*}
[i, i, i, i]_{k}=2\theta+\sum_{(v_1, v_2)\in \langle k\rangle^2}
(v_2-v_1, -v_1)_{k}(0, v_1-i)_{k} (0, v_2-i)_{k},
\end{equation*}
where
\begin{equation*}
\theta=
\begin{cases}
(0, -i)_{k} f, & \mbox{if} \ f \ \mbox{is even}, \\
0, & \mbox{if} \ f \ \mbox{is odd}
\end{cases}
\end{equation*}
as required. The proof of Lemma \ref{lem2.7} is complete.
\end{proof}

Letting $k=4$ in Lemma \ref{lem2.7}, by Lemma \ref{lem2.3} (1)
and Lemma \ref{lem2.4}, the following result follows immediately.

\begin{lem}\label{exm2.4} We have
\begin{small}
\begin{align*}
&[0, 0, 0, 0]_{4}=
\begin{cases}
\frac{1}{256}\big(q^3-4q^2-79q-60sq-20s^2
      -60s-34\big), & \mbox{if }  q\equiv 1 \pmod 8, \\
\frac{1}{256}\big(q^3-4q^2-28sq+25q+12s^2
      -12s-10\big), & \mbox{if }  q\equiv 5 \pmod 8,
\end{cases}
\\
&[1, 1, 1, 1]_{4}=
\begin{cases}
\frac{1}{256}\big(q^3-4q^2+20sq-48tq+13q
      -32st-12s+16t^2-48t-18\big), &\!\! \mbox{if }  q\equiv 1 \pmod 8, \\
\frac{1}{256}\big(q^3-4q^2+4sq-16tq-11q
      +32st-12s+16t^2+48t+6\big), & \!\!\mbox{if }  q\equiv 5 \pmod 8,
\end{cases}
\\
&[2, 2, 2, 2]_{4}=
\begin{cases}
\frac{1}{256}\big(q^3-4q^2+20sq+13q
      -12s-48t^2-18\big), & \mbox{if }  q\equiv 1 \pmod 8, \\
\frac{1}{256}\big(q^3-4q^2+20sq+21q
      -60s+80t^2+6\big), & \mbox{if }  q\equiv 5 \pmod 8,
\end{cases}
\\
&[3, 3, 3, 3]_{4}=
\begin{cases}
\frac{1}{256}\big(q^3-4q^2+20sq+48tq+13q
      +32st-12s+16t^2+48t-18\big), &\!\! \mbox{if }  q\equiv 1 \pmod 8, \\
\frac{1}{256}\big(q^3-4q^2+4sq+16tq-11q
      -32st-12s+16t^2-48t+6\big), & \!\!\mbox{if }  q\equiv 5 \pmod 8,
\end{cases}
\end{align*}
\end{small}
where $s$ and $t$ are given as in (\ref{st}).
\end{lem}
\begin{proof}
One can use Lemmas
\ref{lem2.4} and \ref{lem2.7} (ii) and some computations
to get the desired results. The proof of Lemma \ref{exm2.4}
is complete.
\end{proof}

\section{Values of $N_n(c)$ for $1\le n\le 4$}
In this section, we compute $N_n(c)$ for $1\le n\le 4$.
We begin with the following lemma which expresses the number
of solutions of a diagonal equation in terms of dimension
$n$ cyclotomic number of order $k$.

\begin{lem}\label{lem2.8}
Let $c\in\mathbb{F}_q^*=\langle g\rangle$ and $q-1=kf$.
Let ${N'}_{n}^{(k)}(c)$ denote the number of zeros
$(x_1, ..., x_n)\in(\mathbb{F}_q^*)^n$ of the diagonal
equation
\begin{equation}\label{eq2.14}
x_1^k+\cdots +x_n^k=c.
\end{equation}
Then
\begin{align*}
{N'}_{n}^{(k)}(c)=k^n \big[\underbrace{\ind_g\Big(\frac{1}{c}\Big), ...,\ind_g\Big(\frac{1}{c}\Big)}_{n \ \mbox{times}}\big]_{k}.
\end{align*}
\end{lem}

\begin{proof}
Set $i=\ind_g(\frac{1}{c})$. Then we can rewrite (\ref{eq2.14})
in the following form
\begin{equation*}
g^{k\ind_g(x_1)+i}+\cdots +g^{k\ind_g(x_n)+i}=1.
\end{equation*}
Hence
\begin{align*}
{N'}_{n}^{(k)}(c)=\sharp\{(u_1, ..., u_n)\in\langle q\rangle^n:
g^{ku_1+i}+\cdots + g^{ku_n+i}=1\}.
\end{align*}

Since $q-1=kf$, $ku\equiv kv \pmod {q-1}$ if and only if
$u\equiv v\pmod f$. So $g^{ku+i}=g^{k(u+jf)+i}$ for all
integers $j$ with $0\leq j\leq k-1$. Therefore
\begin{equation*}
{N'}_{n}^{(k)}(c)=\sharp\{(u_1, ..., u_n)\in\langle f\rangle^n:
g^{k(u_1+j_1f)+i}+\cdots + g^{k(u_n+j_nf)+i}=1\}
\end{equation*}
for all $0\leq j_1, \dots, j_n \leq k-1$. For any
$j_r (1\leq r\leq n)$, there exists $k$ choices,
and so by the definition of dimension $n$ cyclotomic
numbers of order $k$, one can obtain that
\begin{align*}
{N'}_{n}^{(k)}(c)=& k^n\sharp\{(u_1, ..., u_n)\in\langle f\rangle^n:
g^{ku_1+i}+\cdots + g^{ku_n+i}=1\}\\
=& k^n[\underbrace{i, \dots, i}_{n\ \rm{times}}]_{k}
= k^n\big[\underbrace{\ind_g\Big(\frac{1}{c}\Big), ...,
\ind_g\Big(\frac{1}{c}\Big)}_{n\ \rm{times}}\big]_{k}
\end{align*}
as desired.

This concludes the proof of Lemma \ref{lem2.8}.
\end{proof}

For any $c\in\mathbb{F}_q^*$, the next lemma exhibits
the relation between $N_n^{(k)}(c)$ and $N_n^{(k)}(g^a)$,
where $a\equiv \ind_g(c) \pmod k$.

\begin{lem}\label{lem2.9}
Let $c\in\mathbb{F}_q^*=\langle g\rangle$ and
$\ind_g(c)\equiv a\pmod k$ with $a\in\langle k\rangle$.
Then $N_n^{(k)}(c)=N_n^{(k)}(g^a)$.
\end{lem}

\begin{proof}
Since $a\equiv \ind_g(c) \pmod k$, one has
$c=g^{\ind_g(c)}=g^{ku+a}$ for some integer $u$.
Then $x_1^k +\dots +x_n^k =c$ can be rewritten as
$(g^{-u}x_1)^k +\dots +(g^{-u}x_n)^k=g^a$.
Therefore $N_n^{(k)}(c)=N_n^{(k)}(g^a)$ as desired.

This finishes the proof of Lemma \ref{lem2.9}.
\end{proof}

In the next lemma, we will derive the number of zeros of
diagonal quartic form by the cyclotomic theory. These give
the initial values of $N_n(c)$ with $c\in\mathbb{F}_q^*$.

\begin{lem}\label{lem2.10}
Let $\mathbb{F}_q$ be the finite field of $q=p^m\equiv 1\pmod 4$
elements with $p$ being an odd prime and $\mathbb{F}_q^*=
\langle g\rangle$. Let $c\in\mathbb{F}_q^*$ and let $N_n(c)$
be the number of zeros of the following diagonal quartic equation
\begin{equation*}
x_1^4+...+x_n^4=c.
\end{equation*}
Then each of the following is true:

{\rm (i)}.
\begin{equation*}
N_1(c)=
\begin{cases}
4, & \mbox{if } \ind_g(c)\equiv 0 \pmod  4, \\
0, & \mbox{if } \ind_g(c)\not\equiv 0\pmod  4.
\end{cases}
\end{equation*}

{\rm (ii)}.
\begin{equation*}
N_2(c)=q+
\begin{cases}
-3+\epsilon_1(c), & \mbox{if } q\equiv 1\pmod  8, \\
1+\epsilon_2(c), & \mbox{if } q\equiv 5\pmod  8,
\end{cases}
\end{equation*}
where
\small
\begin{equation}\label{eps2}
\epsilon_1(c)=
\begin{cases}
-6s, & \mbox{if } \ind_g(c)\equiv 0\ \pmod 4, \\
2s+8t, & \mbox{if } \ind_g(c)\equiv 1\ \pmod 4, \\
2s, & \mbox{if } \ind_g(c)\equiv 2\ \pmod 4, \\
2s-8t, & \mbox{if } \ind_g(c)\equiv 3\ \pmod 4,
\end{cases}
~~~~
\epsilon_2(c)=
\begin{cases}
2s, & \mbox{if } \ind_g(c)\equiv 0\ \pmod 4, \\
2s-8t, & \mbox{if } \ind_g(c)\equiv 1\ \pmod 4, \\
-6s, & \mbox{if } \ind_g(c)\equiv 2\ \pmod 4, \\
2s+8t, & \mbox{if } \ind_g(c)\equiv 3\ \pmod 4.
\end{cases}
\end{equation}
\small

{\rm (iii)}.
\begin{equation*}
N_3(c)=q^2+6s+
\begin{cases}
\epsilon_3(c), & \mbox{if } q\equiv 1\ \pmod 8,\\
\epsilon_4(c), & \mbox{if } q\equiv 5\ \pmod 8,
\end{cases}
\end{equation*}
where
\small\begin{equation}\label{eps3}
\epsilon_3(c)=
\begin{cases}
17q+4s^2, & \mbox{if } \ind_g(c)\equiv 0\ \pmod 4, \\
-7q-8st, & \mbox{if } \ind_g(c)\equiv 1\ \pmod 4, \\
-7q+16t^2, & \mbox{if } \ind_g(c)\equiv 2\ \pmod 4, \\
-7q+8st, & \mbox{if } \ind_g(c)\equiv 3\ \pmod 4,
\end{cases}
~~
\epsilon_4(c)=
\begin{cases}
-3q-4s^2, & \mbox{if } \ind_g(c)\equiv 0\ \pmod 4, \\
5q+8st, & \mbox{if } \ind_g(c)\equiv 1\ \pmod 4, \\
-3q-16t^2, & \mbox{if } \ind_g(c)\equiv 2\ \pmod 4, \\
5q-8st, & \mbox{if } \ind_g(c)\equiv 3\ \pmod 4.
\end{cases}
\end{equation}
\small

{\rm (iv)}.
\begin{equation*}
N_4(c)=q^3-4s^2+
\begin{cases}
\epsilon_5(c)-17q, & \mbox{if } q\equiv 1\ \pmod 8,\\
\epsilon_6(c)+7q, & \mbox{if } q\equiv 5\ \pmod 8,
\end{cases}
\end{equation*}
where
\small
\begin{equation}\label{eps4}
\epsilon_5(c)=
\begin{cases}
-60sq, & \mbox{if } \ind_g(c)\equiv 0\ \pmod 4, \\
20sq+48tq, & \mbox{if } \ind_g(c)\equiv 1\ \pmod 4, \\
20sq, & \mbox{if } \ind_g(c)\equiv 2\ \pmod 4, \\
20sq-48tq, & \mbox{if } \ind_g(c)\equiv 3\ \pmod 4,
\end{cases}
~~~~
\epsilon_6(c)=
\begin{cases}
-28sq, & \mbox{if } \ind_g(c)\equiv 0\ \pmod 4, \\
4sq+16tq, & \mbox{if } \ind_g(c)\equiv 1\ \pmod 4, \\
20sq, & \mbox{if } \ind_g(c)\equiv 2\ \pmod 4, \\
4sq-16tq, & \mbox{if } \ind_g(c)\equiv 3\ \pmod 4,
\end{cases}
\end{equation}
\small
with $s$ and $t$ are given as in (\ref{st}).
\end{lem}

\begin{proof}
First of all, let $N'_n(c)$ be the number of zeros of
$x_1^4+\dots +x_n^4=c$ with $x_1, ..., x_n \in\mathbb{F}_q^*$.
Evidently, one has $\ind_g(1)\equiv 0\pmod 4$ and
$\ind_g(\frac{1}{g^j})\equiv 4-j \pmod 4$ for $j\in\{1,2,3\}$.
Since $c\in\mathbb{F}_q^*$, one may let $\ind_g(c)
\equiv i\pmod 4$ with $i\in\{0, 1, 2, 3\}$.

(i). Let $n=1$. By Lemma \ref{lem2.9}, one yields that
$N_1(c)=N_1(g^i)$. Since $c\in\mathbb{F}_q^*$, $x_1^4=c$
has no zero solution $x_1=0$. Thus $N_1(g^i)=N'_1(g^i)$.
By Lemma \ref{lem2.8}, (\ref{eq2.5}) and Remark \ref{rem2.2},
one can deduce that
\begin{equation}\label{eq2.16}
N_1(g^i)=4~\big[\ind_g(\frac{1}{g^i})\big]_{4}=4 [4-i]_4
=\begin{cases}
4, & \mbox{if } i=0, \\
0, & \mbox{if } i\in\{1,2,3\}.
\end{cases}
\end{equation}
Since $N_1(c)=N_1(g^i)$, by (\ref{eq2.16}) we have
\begin{equation*}
N_1(c)=
\begin{cases}
4, & \mbox{if } \ind_g(c)\equiv 0\ \pmod 4, \\
0, & \mbox{if } \ind_g(c)\not\equiv 0\ \pmod 4
\end{cases}
\end{equation*}
as required.

(ii). Let $n=2$. Lemma \ref{lem2.9} gives us that
$N_2(c)=N_2(g^i)$. One can easily check that
\begin{equation}\label{eq2.17}
N_2(c)=N_2(g^i)=N'_2(g^i)+2N'_1(g^i).
\end{equation}
By Lemma \ref{lem2.8} and Remark \ref{rem2.2}, one has
\begin{equation*}
N'_2(g^i)=16 \big[\ind_g(\frac{1}{g^i}),
\ind_g(\frac{1}{g^i})\big]_{4}=16 [4-i, 4-i]_{4}.
\end{equation*}
From Lemma \ref{exm2.2}, one derives that
if $q\equiv 1 \pmod 8$, then
\begin{equation}\label{eq2.18}
N'_2(g^i)=
\begin{cases}
q-6s-11, & \mbox{if } i=0, \\
q+2s+8t-3, & \mbox{if } i=1, \\
q+2s-3, & \mbox{if } i=2, \\
q+2s-8t-3, & \mbox{if } i=3,
\end{cases}
\end{equation}
and if $q\equiv 5 \pmod 8$, then
\begin{equation}\label{eq2.19}
N'_2(g^i)=
\begin{cases}
q+2s-7, & \mbox{if } i=0, \\
q+2s-8t+1, & \mbox{if } i=1, \\
q-6s+1, & \mbox{if } i=2, \\
q+2s+8t+1, & \mbox{if } i=3.
\end{cases}
\end{equation}

Noticing that $q=s^2+4t^2$, putting (\ref{eq2.16}), (\ref{eq2.18})
and (\ref{eq2.19}) into (\ref{eq2.17}) yields that
\begin{equation*}
N_2(c)=
\begin{cases}
q+\epsilon_1(c)-3, & \mbox{if } q\equiv 1\ \pmod 8, \\
q+\epsilon_2(c)+1, & \mbox{if } q\equiv 5\ \pmod 8,
\end{cases}
\end{equation*}
where $\epsilon_1$ and $\epsilon_2$ are determined
by (\ref{eps2}). Part (ii) is proved.

(iii). Let $n=3$. Then in the similar way as in part (ii),
\begin{equation}\label{eq2.20}
N_3(c)=N_3(g^i)=N'_3(g^i)+3N'_2(g^i)+3N'_1(g^i).
\end{equation}
By Lemma \ref{lem2.8} and Remark \ref{rem2.2},
\begin{equation*}
N'_3(g^i)=64~\big[\ind_g(\frac{1}{g^i}),\ \ind_g(\frac{1}{g^i}),\
\ind_g(\frac{1}{g^i})\big]_{4}=64 [4-i, 4-i, 4-i]_4.
\end{equation*}
From Lemma \ref{exm2.3}, one gets that if $q\equiv 1 \pmod 8$, then
\begin{equation}\label{eq2.21}
N'_3(g^i)=
\begin{cases}
q^2+14q+4s^2+24s+21, & \mbox{if } i=0, \\
q^2-10q-8st-24t+9, & \mbox{if } i=1, \\
q^2-6q-4s^2+9, & \mbox{if } i=2, \\
q^2-10q+8st+24t+9, & \mbox{if } i=3,
\end{cases}
\end{equation}
and if $q\equiv 5 \pmod 8$, then
\begin{equation}\label{eq2.22}
N'_3(g^i)=
\begin{cases}
q^2-6q-4s^2+9, & \mbox{if } i=0, \\
q^2+2q+8st+24t-3, & \mbox{if } i=1, \\
q^2-6q+24s-16t^2-3, & \mbox{if } i=2, \\
q^2+2q-8st-24t-3, & \mbox{if } i=3.
\end{cases}
\end{equation}
Applying (\ref{eq2.16}), (\ref{eq2.18}) to (\ref{eq2.22}), one has
\begin{equation*}
N_3(c)=
\begin{cases}
q^2+6s+\epsilon_3(c), & \mbox{if } q\equiv 1\ \pmod 8,\\
q^2+6s+\epsilon_4(c), & \mbox{if } q\equiv 5\ \pmod 8,
\end{cases}
\end{equation*}
where $\epsilon_3$ and $\epsilon_4$ are determined
by (\ref{eps3}). Part (iii) is proved.

(iv). Let $n=4$. In the similar way as in parts (ii)
and (iii), we have
\begin{equation}\label{eq2.23}
N_4(c)=N_4(g^i)=N'_4(g^i)+4N'_3(g^i)+6N'_2(g^i)+4N'_1(g^i).
\end{equation}
By Lemma \ref{lem2.8} and Remark \ref{rem2.2},
\begin{align*}
N'_4(g^i)=&256~\big[\ind_g(\frac{1}{g^i}),\ \ind_g(\frac{1}{g^i}),
\ \ind_g(\frac{1}{g^i}),\ \ind_g(\frac{1}{g^i})\big]_{4}\\
=&256 [4-i, 4-i, 4-i, 4-i]_{4}.
\end{align*}
Then from Lemma \ref{exm2.4}, it follows that
if $q\equiv 1 \pmod 8$, then
\small
\begin{equation}\label{eq2.24}
N'_4(g^i)=
\begin{cases}
q^3-4q^2-60sq-79q-20s^2-60s-34, & \mbox{if } i=0, \\

q^3-4q^2+20sq+48tq+13q+32st-12s+16t^2+48t-18, & \mbox{if } i=1,\\

q^3-4q^2+20sq+13q-12s-48t^2-18, & \mbox{if } i=2,\\

q^3-4q^2+20sq-48tq+13q-32st-12s+16t^2-48t-18, & \mbox{if } i=3,
\end{cases}
\end{equation}
\small
and if $q\equiv 5 \ \pmod 8$, then
\begin{equation}\label{eq2.25}
N'_4(g^i)=
\begin{cases}
q^3-4q^2-28sq+25q+12s^2-12s-10, & \mbox{if } i=0,\\
q^3-4q^2+4sq+16tq-11q-32st-12s+16t^2-48t+6, & \mbox{if } i=1,\\
q^3-4q^2+20sq+21q-60s+80t^2+6, & \mbox{if } i=2,\\
q^3-4q^2+4sq-16tq-11q
+32st-12s+16t^2+48t+6, & \mbox{if } i=3.
\end{cases}
\end{equation}

Finally, applying (\ref{eq2.16}), (\ref{eq2.18}), (\ref{eq2.19}),
(\ref{eq2.21}) to (\ref{eq2.25}) and noticing that $q=s^2+4t^2$,
\begin{equation*}
N_4(c)=
\begin{cases}
q^3-17q-4s^2+\epsilon_5(c), & \mbox{if } q\equiv 1\ \pmod 8,\\
q^3+7q-4s^2+\epsilon_6(c), & \mbox{if } q\equiv 5\ \pmod 8,
\end{cases}
\end{equation*}
where $\epsilon_5$ and $\epsilon_6$ are determined by (\ref{eps4}).
Part (iv) is proved.

This concludes the proof of Lemma \ref{lem2.10}.
\end{proof}

\section{Proof of Theorem \ref{thm1.2}}

In this section, we present the proof of Theorem \ref{thm1.2}.
For any $x\in\mathbb{F}_q$ with $q=p^m$, we define the
{\it trace} of $x$ relative to $\mathbb{F}_p$,
denoted by $\Tr(x)$, as follows:
\begin{equation*}
\Tr(x):=x+x^{p}+\cdots +x^{p^{m-1}}.
\end{equation*}
Clearly, the trace function is a linear mapping
from $\mathbb{F}_q$ onto $\mathbb{F}_p$ (see.
for instance, \cite{[LN]}).
The following lemma is well known in finite fields.

\begin{lem}\label{lem2.11}\cite{[Ap]}
Let $y$ be any element of $\mathbb{F}_q$. Then
\begin{equation*}
\sum_{x\in\mathbb{F}_q}\exp\Big(\frac{2\pi{\rm i}\Tr(xy)}{p}\Big)
=\begin{cases}
q, & \mbox{if } \ y=0, \\
0, & \mbox{if } \ y\neq 0.
\end{cases}
\end{equation*}
\end{lem}

For any $u\in\mathbb{F}_q^*$, let
$T_u:=\sum_{v\in\mathbb{F}_q}\exp\big(\frac{2\pi
{\rm i}\Tr(uv^4)}{p}\big)$. The following result
is well known and due to Myerson.

\begin{lem}\label{lem2.12}\cite{[My1]}
Let $\mathbb{F}_q$ be the finite field of $q=p^m=4f+1$
elements with $p$ an odd prime. Let $g$ be a generator
of $\mathbb{F}_q^*$. Then $T_1$, $T_g$, $T_{g^2}$ and
$T_{g^3}$ are the roots of the equation
\begin{align*}
x^4-6qx^2+8qsx+q^2-4qs^2=0\ \ \ if \ q\equiv 1\ \pmod 8,\\
x^4+2qx^2+8qsx+9q^2-4qs^2=0\ \ \ if \ q\equiv 5\ \pmod 8,
\end{align*}
where $s$ is uniquely determined by $q=s^2+4t^2$,
$s\equiv 1\ \pmod 4$, and if $p\equiv 1\ \pmod 4$,
then $\gcd(s,\ p)=1$.
\end{lem}

We can now give the proof of Theorem \ref{thm1.2}.\\

{\it Proof of Theorem \ref{thm1.2}.}
(i). By Lemma \ref{lem2.11}, one has
\begin{align}\label{eq3.1}
N_n(c)&=\frac{1}{q}\sum_{x\in\mathbb{F}_q}~~~
\sum_{(x_1,\dots,x_n)\in\mathbb{F}_q^n}
\exp\Big(\frac{2\pi{\rm i}\Tr\big(x(x_1^4+\dots+x_n^4-c)\big)}
{p}\Big)\notag\\
&=\frac{1}{q}\sum_{x\in\mathbb{F}_q}
\Big(\sum_{y\in\mathbb{F}_q}\exp\Big(\frac{2\pi{\rm i}\Tr(xy^4)}{p}\Big)\Big)^n
\exp\Big(\frac{2\pi{\rm i}\Tr(-xc)}{p}\Big)\notag\\
&=q^{n-1}+\frac{1}{q}\sum_{x\in\mathbb{F}_q^*}\bigg(\sum_{y\in\mathbb{F}_q}
\exp\Big(\frac{2\pi{\rm i}\Tr(xy^4)}{p}\Big)\bigg)^n\exp\big(\frac{2\pi{\rm i}\Tr(-xc)}{p}\big)\notag\\
&:=q^{n-1}+\frac{1}{q}\ R(n,c).
\end{align}

Now we turn our attention to the evaluation of $R(n,c)$. Recalling that $$T_u=\sum_{y\in\mathbb{F}_q}\exp\big(\frac{2\pi {\rm i}\Tr(uy^4)}{p}\big),$$
we have
\begin{align}\label{eq3.2}
R(n,c)=&\sum_{x\in\mathbb{F}_q^*}T_x^n\exp\big(\frac{2\pi{\rm i}\Tr(-xc)}{p}\big)
=\sum_{l=0}^3 \sum_{\substack{x\in\mathbb{F}_q^*\\ \ind_g(x)\equiv l\pmod 4}}T_x^n\exp\big(\frac{2\pi{\rm i}\Tr(-xc)}{p}\big).
\end{align}
Let $\ind_g(x)\equiv l\pmod 4$ with $0\le l\le 3$.
Then $x=g^{4k+l}$ for some integer $k$ and
\begin{align*}
T_x&=\sum_{y\in\mathbb{F}_q}\exp\Big(\frac{2\pi{\rm i}\Tr(g^{4k+l}y^4)}{p}\Big)
=\sum_{y\in\mathbb{F}_q}\exp\Big(\frac{2\pi{\rm i}\Tr\big(g^l(g^{k}y)^4\big)}{p}\Big).
\end{align*}
Note that $g^{k}y$ runs through $\mathbb{F}_q$ if $y$ runs over $\mathbb{F}_q$. Thus
\begin{equation}\label{eq3.3}
T_x=\sum_{y\in\mathbb{F}_q}\exp\Big(\frac{2\pi{\rm i}\Tr(g^{l}y^4)}{p}\Big)=T_{g^l}.
\end{equation}
With (\ref{eq3.3}) applied to (\ref{eq3.2}), one has
\begin{align}\label{eq3.4}
R(n,c)=\sum_{l=0}^{3}T_{g^l}^{n}\lambda_l(c),
\end{align}
where
\begin{equation*}
\lambda_l(c):=\sum_{\substack{x\in\mathbb{F}_q^*\\
\ind_g(x)\equiv l(\bmod4)}}\exp\big(\frac{2\pi{\rm i}\Tr(-xc)}{p}\big).
\end{equation*}

The proof is divided into the following two cases.

{\sc Case 1.} $q\equiv 1\ \pmod 8$. Then for $0\le l\le 3$,
by Lemma \ref{lem2.11}, one has
\begin{equation}\label{eq3.5}
T_{g^l}^4-6qT_{g^l}^2+8qsT_{g^l}+q^2-4qs^2=0.
\end{equation}
Multiplying $T_{g^l}^{n-4}\lambda_l(c)$ on both sides of
(\ref{eq3.5}), one arrives at
\begin{equation*}
T_{g^l}^n\lambda_{l}(c)-6qT_{g^l}^{n-2}\lambda_{l}(c)+8qsT_{g^l}^{n-3}
\lambda_{l}(c)+(q^2-4qs^2)T_{g^l}^{n-4}\lambda_{l}(c)=0.
\end{equation*}
Then taking the sum gives us that
\begin{align}\label{eq3.6}
\sum_{l=0}^{3}T_{g^l}^n\lambda_{l}(c)-6q\sum_{l=0}^{3}T_{g^l}^{n-2}\lambda_{l}(c)
+8qs\sum_{l=0}^{3}T_{g^l}^{n-3}\lambda_{l}(c)+(q^2-4qs^2)
\sum_{l=0}^{3}T_{g^l}^{n-4}\lambda_{l}(c)=0.
\end{align}

For $n\ge 5$, by (\ref{eq3.4}) and (\ref{eq3.6}), we have
\begin{align}\label{eq3.7}
R(n,c)-6qR(n-2,c)+8qsR(n-3,c)+(q^2-4qs^2)R(n-4,c)=0.
\end{align}
It then follows from (\ref{eq3.7}) that
\begin{align*}
&(1-6qx^2+8qsx^3+(q^2-4qs^2)x^4)\sum_{n=1}^{\infty}R(n,c)x^n\\
=&\sum_{n=1}^{\infty}R(n,c)x^n-6q\sum_{n=1}^{\infty}R(n,c)x^{n+2}+
8qs\sum_{n=1}^{\infty}R(n,c)x^{n+3}+(q^2-4qs^2)\sum_{n=1}^{\infty}R(n,c)x^{n+4}\\
=&\sum_{n=1}^{4}R(n,c)x^n-6q\sum_{n=1}^{2}R(n,c)x^{n+2}+8qsR(1,c)x^{4}\\
&+\sum_{n=5}^{\infty}(R(n,c)-6qR(n-2,c)+8qsR(n-3,c)+(q^2-4qs^2)R(n-4,c))x^n\\
=&R(1,c)(x-6qx^3+8qsx^4)+R(2,c)(x^2-6qx^4)+R(3,c)x^3+R(4,c)x^4.
\end{align*}
Therefore
\begin{align*}
&\sum_{n=1}^{\infty}R(n,c)x^n=
\frac{R(1,c)(x-6qx^3+8qsx^4)+R(2,c)(x^2-6qx^4)+R(3,c)x^3
+R(4,c)x^4}{1-6qx^2+8qsx^3+(q^2-4qs^2)x^4}.
\end{align*}
Then noticing the well-known identity that
$\sum_{n=1}^{\infty}q^{n-1}x^n=\frac{x}{1-qx},$
by (\ref{eq3.1}) we have
\begin{align}\label{eq3.8}
&\sum_{n=1}^{\infty}N_n(c)x^n
=\sum_{n=1}^{\infty}\big(q^{n-1}+\frac{1}{q}R(n,c)\big)x^n\notag\\
=&\frac{x}{1-qx}+\frac{1}{q}\frac{R(1,c)(x-6qx^3+8qsx^4)+R(2,c)(x^2-6qx^4)
+R(3,c)x^3+R(4,c)x^4}{1-6qx^2+8qsx^3+(q^2-4qs^2)x^4}\notag\\
:=& \frac{x}{1-qx}+\frac{\tilde B_1(c,x)}{1-6qx^2+8qsx^3+(q^2-4qs^2)x^4}.
\end{align}
Since
\begin{align*}
\tilde B_1(c,x)=&(N_1(c)-1)(x-6qx^3+8qsx^4)+(N_2(c)-q)(x^2-6qx^4)\\
&+(N_3(c)-q^2)x^3+(N_4(c)-q^3)x^4\\
=&(N_1(c)-1)x+(N_2(c)-q)x^2+(-6qN_1(c)+N_3(c)+6q-q^2)x^3\\
&+(8qsN_1(c)-6qN_2(c)+N_4(c)-8qs+6q^2-q^3)x^4,
\end{align*}
it follows from Lemma \ref{lem2.10} that
$$
\tilde B_1(c,x)=6sx^3+(q-4s^2)x^4+B_1(c,x),
$$
where $B_1(c,x)$ is given in (\ref{eq1.2}).
So the desired result follows immediately.

{\sc Case 2.} $q\equiv 5\ \pmod 8$.
For $0\le l\le 3$, by Lemma \ref{lem2.12} one has
\begin{equation*}
T_{g^l}^4+2qT_{g^l}^2+8qsT_{g^l}+9q^2-4qs^2=0.
\end{equation*}
In the same way as in {\sc case 1}, we can derive that if
$n\ge 5$, then
\begin{align}\label{eq3.10}
R(n,c)+2qR(n-2,c)+8qsR(n-3,c)+(9q^2-4qs^2)R(n-4,c)=0.
\end{align}
It then follows from (\ref{eq3.10}) that
\begin{align}\label{eq3.11'}
&\big(1+2qx^2+8qsx^3+(9q^2-4qs^2)x^4\big)\sum_{n=1}^{\infty}R(n,c)x^n\notag\\
=&\sum_{n=1}^{\infty}R(n,c)x^n+2q\sum_{n=1}^{\infty}R(n,c)x^{n+2}+8qs\sum_{n=1}^{\infty}R(n,c)x^{n+3}
+(9q^2-4qs^2)\sum_{n=1}^{\infty}R(n,c)x^{n+4} \notag\\
=&\sum_{n=1}^{4}R(n,c)x^n+2q\sum_{n=1}^{2}R(n,c)x^{n+2}+8qs R(1,c)x^{4}\notag\\
&+ \sum_{n=5}^{\infty}\big(R(n,c)+2qR(n-2,c)+8qsR(n-3,c)+(9q^2-4qs^2)R(n-4,c)\big)x^n\notag\\
=&R(1,c)(x+2qx^3+8qsx^4)+R(2,c)(x^2+2qx^4)+R(3,c)x^3+R(4,c)x^4.
\end{align}

Since $\sum_{n=1}^{\infty}q^{n-1}x^n=\frac{x}{1-qx}$, by (\ref{eq3.1})
and (\ref{eq3.11'}) one can deduce that
\begin{align}\label{eq3.11}
\sum_{n=1}^{\infty}N_n(c)x^n=&\sum_{n=1}^{\infty}
\big(q^{n-1}+\frac{1}{q}R(n,c)\big)x^n\notag\\
=&\frac{x}{1-qx}+\frac{\tilde B_2(c,x)}
{1+2qx^2+8qsx^3+(9q^2-4qs^2)x^4},
\end{align}
where
\begin{align*}\tilde B_2(c,x)
:=&(N_1(c)-1)(x+2qx^3+8qsx^4)+(N_2(c)-q)(x^2+2qx^4)\\
&+(N_3(c)-q^2)x^3+(N_4(c)-q^3)x^4.
\end{align*}
By Lemma \ref{lem2.10}, one has
$$
\tilde B_2(c,x)=6sx^3+(9q-4s^2)x^4+B_2(c,x),
$$
where $B_1(c,x)$ is given in (\ref{eq1.3'}).
Then the expected result follows immediately
from (\ref{eq3.11}).

(ii). Let $q\equiv 3\pmod 4$. It is well known
that (see, for example, \cite{[LN]})
\begin{equation}\label{eq1.2'}
N(x^k=c)=\sum_{j=0}^{d-1}\varphi^{j}(c),
\end{equation}
where $\varphi$ is a multiplicative character of
$\mathbb{F}_q$ of order $d=\gcd(k, q-1)$.
Since $q\equiv 3\pmod 4$, by (\ref{eq1.2'})
$N(x^4=c)=N(x^2=c)$. It then follows that
\begin{align*}
N(x_1^4+...+x_n^4=c)&=\sum_{(x_1, ..., x_n)
\in\mathbb{F}_q^n\atop x_1^4+...+x_n^4=c} 1
=\sum_{(x_1, ..., x_n)\in\mathbb{F}_q^n
\atop c_1+...+c_n=c}N(x_1^4=c_1)\cdots N(x_n^4=c_n)\\
&=\sum_{(x_1, ..., x_n)\in\mathbb{F}_q^n\atop
c_1+...+c_n=c}N(x_1^2=c_1)\cdots N(x_n^2=c_n)\\
&=\sum_{(x_1, ..., x_n)\in\mathbb{F}_q^n\atop
x_1^2+...+x_n^2=c} 1
=N(x_1^2+\dots +x_n^2=c).
\end{align*}
Then by the formula for $N(x_1^2+...+x_n^2=c)$
given in \cite{[LN]}, we obtain that
\begin{equation}\label{518eq1}
N_n(c)=
\begin{cases}
q^{n-1}+v(c)q^{\frac{n-2}{2}}\eta\big((-1)^{\frac{n}{2}}\big),
& {\rm if }\ n \ {\rm is\ even}, \\
q^{n-1}+q^{\frac{n-1}{2}}\eta\big((-1)^{\frac{n-1}{2}}c\big),
& {\rm if }\ n \ {\rm is\ odd},
\end{cases}
\end{equation}
where the integer-valued function $v$ on $\mathbb{F}_q$
is defined by $v(c):=-1$ if $c\in\mathbb{F}_q^*$ and
$v(0):=q-1$, and $\eta$ is the quadratic character
of $\mathbb{F}_q$.

Since $q\equiv 3\pmod 4$, one has $\eta(-1)=-1$.
By (\ref{518eq1}), we have
\begin{equation}\label{518eqcsz}
N_1(c)-q^0=
\begin{cases}
0, & \mbox{if } c=0, \\
1, & \mbox{if } c \text{\ is\ square}, \\
-1, & \mbox{if } c \text{\ is\ non-square}, \\
\end{cases}
N_2(c)-q=
\begin{cases}
1-q, & \mbox{if } c=0, \\
1, & \mbox{if } c \text{\ is\ square}, \\
1, & \mbox{if } c \text{\ is\ non-square}, \\
\end{cases}
\end{equation}
and
$N_{n+2}(c)-q^{n+1}=-q(N_{n}(c)-q^{n-1})$
for all $c\in \mathbb F_q$ and positive integer $n$.
It then follows that
\begin{align}\label{518eq2}
&(1+qx^2)\sum_{n=1}^{\infty}(N_n(c)-q^{n-1})x^n\notag\\
=&(N_1(c)-1)x+(N_2(c)-q^{})x^2+\sum_{n=3}^{\infty}
\big(N_n(c)-q^{n-1}+q(N_{n-2}(c)-q^{n-3})\big)x^n\notag\\
=&(N_1(c)-1)x+(N_2(c)-q^{})x^2.
\end{align}

Putting (\ref{518eqcsz}) into (\ref{518eq2}) gives us that
\begin{align*}
\sum_{n=1}^{\infty}N_n(c)x^n=
\begin{cases}
\frac{x}{1-qx}+\frac{(1-q)x^2}{1+qx^2}, & \mbox{if } c=0, \\
\frac{x}{1-qx}+\frac{x+x^2}{1+qx^2}, & \mbox{if } c \text{\ is\ square},\\
\frac{x}{1-qx}+\frac{-x+x^2}{1+qx^2}, & \mbox{if } c \text{\ is\ non-square}. \\
\end{cases}
\end{align*}
So part (ii) is proved.

This concludes the proof of Theorem \ref{thm1.2}. \hfill$\Box$

\section{Proof of Theorem \ref{thm1.3}}

In this section, we present the proof of Theorem \ref{thm1.3}.\\

{\it Proof of Theorem \ref{thm1.3}.}
(i). Let $q\equiv 1\pmod 4$. Since $g$ is a generator of $\mathbb{F}_q^*$,
one has $-1=g^{\frac{q-1}{2}}$. Thus $-1$ is quartic if
$q\equiv 1\pmod 8$ and non-quartic if $q\equiv 5\pmod 8$. It follows that
\begin{align*}
M_n(y)
=\sum_{(x_1, ..., x_n)\in\mathbb{F}_q^n\atop x_1^4+...+yx_n^4=0}1
&=\sum_{(x_1, ..., x_{n-1})\in\mathbb{F}_q^{n-1}\atop x_1^4+...+x_{n-1}^4=0}1+
\sum_{x_n\in\mathbb{F}_q^*}\sum_{(x_1, ..., x_{n-1})\in\mathbb{F}_q^{n-1}\atop (\frac{x_1}{x_n})^4+...+(\frac{x_{n-1}}{x_n})^4=(-1)^{\frac{q-1}{4}}y}1\\
&=N_{n-1}(0)+\sum_{x_n\in\mathbb{F}_q^*}\sum_{(x_1, ..., x_{n-1})\in\mathbb{F}_q^{n-1}
\atop x_1^4+...+x_{n-1}^4=(-1)^{\frac{q-1}{4}}y}1\\
&=N_{n-1}(0)+(q-1)N_{n-1}((-1)^{\frac{q-1}{4}}y).
\end{align*}
Hence
\begin{equation}\label{eq3.13}
\sum_{n=1}^{\infty}M_{n+1}(y)x^n=\sum_{n=1}^{\infty}N_n(0)x^n
+(q-1)\sum_{n=1}^{\infty}N_n((-1)^{\frac{q-1}{4}}y)x^n.
\end{equation}
By (\ref{eq1.1}), Theorem \ref{thm1.2} and (\ref{eq3.13}),
one can deduce the desired result for
$\sum_{n=1}^{\infty}M_{n+1}(y)x^n$.

(ii). Let $q\equiv 3\pmod 4$. Then by $-1=g^{\frac{q-1}{2}}$,
we know that $-1$ is non-square. In the same way as in the proof of
part (ii) of Theorem 1.1, we have
\begin{align*}
M_n(y)=&\sum_{(x_1, ..., x_{n})\in\mathbb{F}_q^{n}\atop
x_1^4+...+yx_n^4=0}1
= \sum_{(x_1, ..., x_{n})
\in\mathbb{F}_q^{n}\atop x_1^2+...+yx_n^2=0}1\\
=& \sum_{(x_1, ..., x_{n-1})\in\mathbb{F}_q^{n-1}\atop
x_1^2+...+x_{n-1}^2=0}1+\sum_{x_n\in\mathbb{F}_q^*}
\sum_{(x_1, ..., x_{n-1})\in\mathbb{F}_q^{n-1}\atop
x_1^2+...+x_{n-1}^2=-y}1\\
=& N_{n-1}(0) +(q-1)N_{n-1}(-y)\\
=& N_{n-1}(0) +(q-1)N_{n-1}(1).
\end{align*}
Therefore
\begin{align*}
\sum_{n=1}^{\infty}M_{n+1}(y)x^n=&\sum_{n=1}^{\infty}N_n(0)x^n
+(q-1)\sum_{n=1}^{\infty}N_n(1)x^n\\
=&\frac{x}{1-qx}+\frac{(1-q)x^2}{1+qx^2}+(q-1)
\Big(\frac{x}{1-qx}+\frac{x+x^2}{1+qx^2}\Big)\\
=&\frac{qx}{1-qx}+\frac{(q-1)x}{1+qx^2}
\end{align*}
as expected. The proof of Theorem \ref{thm1.3} is complete.
\hfill$\Box$
\bibliographystyle{amsplain}

\end{document}